\documentclass[12pt,reqno]{amsart}

\usepackage{color}
\usepackage{amsmath, amsthm, amssymb}
\usepackage{amsfonts}
\usepackage[ansinew]{inputenc}
\usepackage[dvips]{epsfig}
\usepackage{graphicx}
\usepackage[english]{babel}
\usepackage{hyperref}
\theoremstyle{plain}
\newtheorem{thm}{Theorem}[section]
\newtheorem{cor}[thm]{Corollary}
\newtheorem{lem}[thm]{Lemma}

\theoremstyle{definition}

\theoremstyle{remark}

\numberwithin{equation}{section}

\newcommand{\average}{{\mathchoice {\kern1ex\vcenter{\hrule height.4pt
width 6pt depth0pt} \kern-9.7pt} {\kern1ex\vcenter{\hrule
height.4pt width 4.3pt depth0pt} \kern-7pt} {} {} }}

\def\R{\mathbb{R}}

\begin{document}

\title[Boundary Harnack for nonlocal operators in non-divergence form]
{The boundary Harnack principle for nonlocal elliptic operators in non-divergence form}

\author{Xavier Ros-Oton}
\address{The University of Texas at Austin, Department of Mathematics, 2515 Speedway, Austin, TX 78751, USA}
\email{ros.oton@math.utexas.edu}

\author{Joaquim Serra}
\address{ETH Z\"urich, Department of Mathematics, Raemistrasse 101, 8092 Z\"urich, Switzerland}
\email{joaquim.serra@math.ethz.ch}


\thanks{XR was supported by NSF grant DMS-1565186. Both authors were supported by MINECO grant MTM2014-52402-C3-1-P (Spain)}

\maketitle

\begin{abstract}
We prove a boundary Harnack inequality for nonlocal elliptic operators $L$ in non-divergence form with bounded measurable coefficients.
Namely, our main result establishes that if $Lu_1=Lu_2=0$ in $\Omega\cap B_1$, $u_1=u_2=0$ in $B_1\setminus\Omega$, and $u_1,u_2\geq0$ in $\R^n$, then $u_1$ and $u_2$ are comparable in $B_{1/2}$.
The result applies to arbitrary open sets $\Omega$.

When $\Omega$ is Lipschitz, we show that the quotient $u_1/u_2$ is H\"older continuous up to the boundary in $B_{1/2}$.
\end{abstract}

\vspace{4mm}

\section{Introduction and results}

The aim of this note is to establish new boundary Harnack inequalities for nonlocal elliptic operators in non-divergence form in general open sets.

To our knowledge, the first boundary Harnack principle for nonlocal elliptic operators was established by Bogdan \cite{Bogdan1}, who proved it for the fractional Laplacian in Lipschitz domains.
Later, his result was extended to arbitrary open sets by Song and Wu in \cite{SW}; see also Bogdan-Kulczycki-Kwasnicki \cite{Bogdan2}.
More recently, Bogdan-Kumagai-Kwasnicki \cite{Bogdan3} established the Boundary Harnack principle in general open sets for a wide class of Markov processes with jumps.
In particular, their results apply to all linear operators of the form
\begin{equation}\label{operator-L-linear}
Lu(x)=\int_{\R^n}\left(\frac{u(x+y)+u(x-y)}{2}-u(x)\right)K(y)\,dy,
\end{equation}
with kernels $K(y)=K(-y)$ satisfying
\begin{equation}\label{ellipt-const-linear}
\qquad \qquad \qquad 0<\frac{\lambda}{|y|^{n+2s}}\leq K(y)\leq \frac{\Lambda}{|y|^{n+2s}},\qquad y\in \R^n;
\end{equation}
see \cite[Example 5.6]{Bogdan3}.

Here, we consider \emph{non-divergence} form operators
\begin{equation}\label{operator-L}
Lu(x)=\int_{\R^n}\left(\frac{u(x+y)+u(x-y)}{2}-u(x)\right)K(x,y)\,dy,
\end{equation}
with kernels $K(x,y)=K(x,-y)$ satisfying
\begin{equation}\label{ellipt-const}
\qquad \qquad \qquad 0<\frac{\lambda}{|y|^{n+2s}}\leq K(x,y)\leq \frac{\Lambda}{|y|^{n+2s}},\qquad x,y\in \R^n.
\end{equation}
No regularity in $x$ is assumed.
These are the nonlocal analogues of second order uniformly elliptic operators $L=\sum_{i,j}a_{ij}(x)\partial_{ij}$ with bounded measurable coefficients; see \cite{BL,S,CS}.

To our knowledge, our results are the first ones that establish boundary Harnack inequalities for such class of nonlocal operators in non-divergence form.
Quite recently, we established in \cite{RS-C1} a boundary Harnack estimate for operators of the form \eqref{operator-L}-\eqref{ellipt-const} under the important extra assumption that $K(x,y)$ is \emph{homogeneous} in $y$.
The results of \cite{RS-C1} are for $C^1$ domains, and the all the proofs are by blow-up and perturbative arguments.
The techniques of the present paper are of very different nature, and completely independent from those in \cite{RS-C1}.

Our first result establishes the boundary Harnack principle in general open sets~$\Omega$, and reads as follows.

\begin{thm}\label{thm-bdryH}
Let $s\in(0,1)$, and $L$ be any operator of the form \eqref{operator-L}-\eqref{ellipt-const}.
Let $\Omega\subset\R^n$ be any open set, with $0\in\partial\Omega$, and $u_1, u_2\in C(B_1)$ be two viscosity solutions of
\begin{equation}\label{pb0}
\left\{ \begin{array}{rcll}
L u_1=Lu_2 &=&0&\textrm{in }B_1\cap \Omega \\
u_1=u_2&=&0&\textrm{in }B_1\setminus\Omega,
\end{array}\right.\end{equation}
satisfying $u_i\geq0$ in $\R^n$ and
\[\int_{\R^n}\frac{u_i(x)}{1+|x|^{n+2s}}\,dx=1.\]
Then,
\[ C^{-1}u_2\leq u_1\leq C\,u_2\qquad\textrm{in}\ B_{1/2}.\]
The constant $C$ depends only on $n$, $s$, $\Omega$, and ellipticity constants.
\end{thm}

Here, the equation $Lu=0$ should be understood in the viscosity sense as $M^+u\geq0\geq M^-u$, where
\[M^+ u=M^+_{\mathcal L_0}u=\sup_{L\in\mathcal L_0} Lu,\qquad
 M^- u=M^-_{\mathcal L_0}u=\inf_{L\in\mathcal L_0} Lu,\]
and $\mathcal L_0$ is the class of operators of the form \eqref{operator-L-linear}-\eqref{ellipt-const-linear}; see \cite{CS} for more details.
The fact that both $u_1$ and $u_2$ solve the \emph{same} equation $Lu_1=Lu_2=0$ can be stated as $M^+(au_1+bu_2)\geq0$ for all $a,b\in\R$.
Notice that taking $a=\pm1$ and $b=0$, or $a=0$ and $b=\pm1$, we get that $M^+u_i\geq0\geq M^-u_i$.

We will in fact prove a more general version of Theorem \ref{thm-bdryH}, in which we allow a right hand side in the equation, $Lu_1=f_1$ and $Lu_2=f_2$ in $\Omega\cap B_1$, with $\|f_i\|_{L^\infty}\leq \delta$, and $\delta>0$ small enough.
In terms of the extremal operators $M^+$ and $M^-$, it reads as follows.

\begin{thm}\label{thm-main}
Let $s\in(0,1)$ and $\Omega\subset\R^n$ be any open set.
Assume that there is $x_0\in B_{1/2}$ and $\varrho>0$ such that $B_{2\varrho}(x_0)\subset \Omega\cap B_{1/2}$.

Then, there exists $\delta>0$, depending only on $n$, $s$, $\varrho$, and ellipticity constants, such that the following statement holds.

Let $u_1,u_2\in C(B_1)$ be viscosity solutions of
\begin{equation}\label{pb}
\left\{ \begin{array}{rcll}
M^+(au_1+bu_2) &\geq&-\delta(|a|+|b|)&\textrm{in }B_1\cap \Omega\\
u_1=u_2&=&0&\textrm{in }B_1\setminus\Omega
\end{array}\right.\end{equation}
for all $a,b\in\R$, and such that
\begin{equation}\label{u-is-nonneg}
u_i\geq0\quad\mbox{in}\quad \R^n, \qquad \int_{\R^n}\frac{u_i(x)}{1+|x|^{n+2s}}\,dx=1.
\end{equation}
Then,
\[ C^{-1}u_2\leq u_1\leq C\,u_2\qquad\textrm{in}\ B_{1/2}.\]
The constant $C$ depends only on $n$, $s$, $\varrho$, and ellipticity constants.
\end{thm}

One of the advantages of Theorem \ref{thm-main} is that it allows us to establish the following result.

\begin{thm}\label{thm-Lip}
Let $s\in(0,1)$ and $\Omega\subset\R^n$ be any Lipschitz domain, with $0\in\partial\Omega$.
Then, there is $\delta>0$, depending only on $n$, $s$, $\Omega$, and ellipticity constants, such that the following statement holds.

Let $u_1,u_2\in C(B_1)$ be viscosity solutions of \eqref{pb} satisfying \eqref{u-is-nonneg}.
Then, there is $\alpha\in(0,1)$ such that
\[ \left\|\frac{u_1}{u_2}\right\|_{C^{0,\alpha}(\overline\Omega\cap B_{1/2})}\leq C.\]
The constants $\alpha$ and $C$ depend only on $n$, $s$, $\Omega$, and ellipticity constants.
\end{thm}

The proof of Theorems \ref{thm-bdryH} and \ref{thm-main} that we present here is quite short and simple, and to our knowledge is new even for the fractional Laplacian $(-\Delta)^s$.
Such proof uses very strongly the nonlocal character of the operator (as it must be! Recall that the boundary Harnack principle is in general false for second order (local) operators in H\"older domains \cite{BHP3}).
Then, we prove Theorem \ref{thm-Lip} by iterating appropriately Theorem \ref{thm-main}.

The paper is organized as follows.
In Section \ref{sec2} we give some preliminaries.
In Section \ref{sec3} we establish Theorems \ref{thm-main} and \ref{thm-bdryH}.
In Section \ref{sec4} we prove Theorem \ref{thm-Lip}.
Finally, in Section \ref{sec5} we extend those results to non-symmetric operators and to operators with drift.

\section{Preliminaries}
\label{sec2}

In this section we recall some results that will be used in our proofs.

An important ingredient to prove our boundary Harnack inequality is the interior Harnack inequality for nonlocal equations in non-divergence form, which states that if $u$ solves
\[M^+u\geq -C_0\qquad \textrm{and}\qquad M^-u\leq C_0\qquad \textrm{in}\quad B_1,\]
and $u\geq0$ in $\R^n$, then
\[\sup_{B_{1/2}}u\leq C\left(\inf_{B_{1/2}}u+C_0\right);\]
see \cite{CS} and also \cite{BL}.

In our proof, in fact, we will need the following two results, which imply the Harnack inequality.
The first one is a half Harnack inequality for subsolutions.

\begin{thm}[\cite{CS3}]\label{half-Harnack-sub}
Assume that $u\in C(B_1)$ satisfies
\[M^+u\geq -C_0\quad \textrm{in}\ B_1\]
in the viscosity sense.
Then,
\[\sup_{B_{1/2}}u\leq C\left(\int_{\R^n}\frac{|u(x)|}{1+|x|^{n+2s}}\,dx+C_0\right).\]
The constant $C$ depends only on $n$, $s$, and ellipticity constants.
\end{thm}

The second one is the other half Harnack inequality, for supersolutions.

\begin{thm} \label{half-Harnack-sup}
Assume that $u\in C(B_1)$ satisfies
\[M^-u\leq C_0\quad \textrm{in}\ B_1\]
in the viscosity sense.
Assume in addition that $u\geq0$ in $\R^n$.
Then,
\[\int_{\R^n}\frac{u(x)}{1+|x|^{n+2s}}\,dx\leq C\left(\inf_{B_{1/2}}u+C_0\right).\]
The constant $C$ depends only on $n$, $s$, and ellipticity constants.
\end{thm}

When $s\geq\frac12$, the result can be found in \cite[Corollary 6.2]{CD}, where it is proved in the more general setting of parabolic and nonsymmetric operators with drift.
For completeness, we give a short proof of Theorem \ref{half-Harnack-sup} here.

\begin{proof}[Proof of Theorem \ref{half-Harnack-sup}]
Let $b\in C^\infty_c(B_{3/4})$ be such that $0\leq b\leq 1$ and $b\equiv1$ in $B_{1/2}$.
Let $t>0$ be the maximum value for which $u\geq tb$.
Notice that $t\leq \inf_{B_{1/2}}u$.
Since $u$ and $b$ are continuous in $B_1$, then there is $x_0\in B_{3/4}$ such that $u(x_0)=tb(x_0)$.

Now, on the one hand, we have
\[M^-(u-tb)(x_0)\leq M^-u(x_0)-tM^-b\leq C_0+Ct.\]
On the other hand, since $u-tb\geq0$ in $\R^n$ and $(u-tb)(x_0)=0$ then
\[M^-(u-tb)(x_0)=\lambda\int_{\R^n}\frac{u(z)-tb(z)}{|x_0-z|^{n+2s}}dz\geq c\int_{\R^n}\frac{u(z)}{1+|z|^{n+2s}}dz-Ct.\]
Combining the previous identities, we get
\[ \inf_{B_{1/2}} u\geq t\geq -c_1C_0+c_2\int_{\R^n}\frac{u(z)}{1+|z|^{n+2s}}dz,\]
and the result follows.
\end{proof}

\section{Proof of Theorem \ref{thm-main}}
\label{sec3}

Theorem \ref{thm-bdryH} is a particular case of Theorem \ref{thm-main}.
We give below the proof of Theorem \ref{thm-main}.
Before that, we need a Lemma.

\begin{lem}\label{lem-use}
Let $s\in(0,1)$ and $\Omega\subset\R^n$ be any open set.
Assume that there is $x_0\in B_{1/2}$ and $\varrho>0$ such that $B_{2\varrho}(x_0)\subset \Omega\cap B_{1/2}$.
Denote $D=B_\varrho(x_0)$.

Let $u\in C(B_1)$ be a viscosity solution of
\[\left\{ \begin{array}{rcll}
M^+u\geq-C_0\qquad\textrm{and}\qquad M^-u &\leq&C_0&\textrm{in }B_1\cap \Omega\\
u&=&0&\textrm{in }B_1\setminus\Omega
\end{array}\right.\]
Assume in addition that $u\geq0$ in $\R^n$.
Then,
\[\sup_{B_{3/4}}u\leq C\left(\inf_{D}u+C_0\right),\]
with $C$ depending only on $n$, $s$, $\varrho$, and ellipticity constants.
\end{lem}

\begin{proof}
Since $u\geq0$ in $B_1$ and $M^+u\geq-C_0$ in $B_1\cap \{u>0\}$, then $M^+u\geq-C_0$ in all of $B_1$.
Thus, by Theorem \ref{half-Harnack-sub} we have
\[\sup_{B_{3/4}}u\leq C\left(\int_{\R^n}\frac{u(x)}{1+|x|^{n+2s}}\,dx+C_0\right).\]
(Notice that Theorem \ref{half-Harnack-sub} gives a the bound in $B_{1/2}$, but by a standard covering argument we get the same in $B_{3/4}$.)
Now, using Theorem~\ref{half-Harnack-sup} in the ball $B_{2\varrho}(x_0)$, we find
\[\int_{\R^n}\frac{u(x)}{1+|x|^{n+2s}}\,dx\leq C\left(\inf_D u+C_0\right),\]
where $D=B_\varrho(x_0)$.
Combining the previous estimates, the Lemma follows.
\end{proof}

We next give the:

\begin{proof}[Proof of Theorem \ref{thm-main}]
First, as in Lemma \ref{lem-use}, by \eqref{u-is-nonneg} we have
\begin{equation}\label{ineq0}
u_i\leq C\quad \textrm{in}\ B_{3/4}
\end{equation}
and
\begin{equation}\label{ineq2}
u_i\geq c>0\quad\textrm{in}\ B_\varrho(x_0),
\end{equation}
provided that $\delta>0$ is small enough.
Notice that $c$ depends on $n$, $s$, ellipticity constants, and $\varrho$, but not on $\Omega$.

Let now $b\in C^\infty_c(B_{1/2})$ be such that $0\leq b\leq 1$ and $b\equiv1$ in $B_{1/4}$, and let $\eta\in C^\infty_c(B_\varrho(x_0))$ such that $0\leq\eta\leq 1$ in $B_\varrho(x_0)$ and $\eta=1$ in $B_{\varrho/2}(x_0)$.
Let
\[w:=u_1\chi_{B_{3/4}}+C_1(b-1)+C_2\eta.\]
Then, thanks to \eqref{ineq0}, if $C_1$ is chosen large enough we will have
\[w\leq 0\quad\textrm{in}\ \R^n\setminus B_{1/2}.\]
Moreover, taking now $C_2$ large enough,
\[\begin{split}
M^+w& \geq M^+u_1+M^-(u_1\chi_{\R^n\setminus B_{3/4}})+C_1M^-b+C_2M^-\eta \\
&\geq -\delta-C-CC_1+cC_2 \geq 1\qquad\qquad\qquad
\textrm{in}\quad \Omega \cap B_{1/2}\setminus B_\varrho(x_0).\end{split}\]
Here we used that $M^+u_1\geq-\delta$ in $\Omega\cap B_1$, that $M^-(u_1\chi_{\R^n\setminus B_{3/4}})\geq -C\int_{\R^n}u_1(x)/(1+|x|^{n+2s})dx\leq C$ in $B_{1/2}$, that $M^-b\geq -C$, and that $M^-\eta\geq c>0$ in $B_1\setminus B_\varrho(x_0)$.
Analogously, for any $C_3\leq \delta^{-1}$ we get that
\[M^+(w-C_3u_2)\geq 1-C_3\delta\geq 0\quad \textrm{in}\ \Omega\cap B_{1/2}\setminus B_\varrho(x_0),\]
Finally, since $w\leq C$ in $B_\varrho(x_0)$ and $u_2\geq c>0$ in $B_\varrho(x_0)$, we clearly have
\[w\leq C_3u_2\quad\textrm{in}\ B_\varrho(x_0)\]
for some big constant $C_3$.
Taking $\delta$ small enough so that $\delta^{-1}\geq C_3$, by the comparison principle we find $w\leq C_3u_2$ in all of $\R^n$.

In particular, since $w\equiv u_1$ in $B_{1/4}\setminus B_\varrho(x_0)$, this yields
\[u_1\leq C_3u_2\quad \textrm{in}\ B_{1/4}\setminus B_\varrho(x_0).\]
Since $u_1$ and $u_2$ are comparable in $B_\varrho(x_0)$, we deduce
\[u_1\leq C u_2\quad \textrm{in}\ B_{1/4},\]
maybe with a bigger constant $C$.
Finally, a standard covering argument yields the same result in $B_{1/2}$, and thus the theorem is proved.
\end{proof}

\section{Proof of Theorem \ref{thm-Lip}}
\label{sec4}

We prove here Theorem \ref{thm-Lip}.
Throughout this section, $\Omega$ will be a Lipschitz domain with $0\in\partial\Omega$.
In particular, there is $\varrho>0$ such that for every $r\in(0,1)$ there is $x_r\in B_{r/2}$ for which
\begin{equation}\label{P}
B_{2\varrho r}(x_r)\subset \Omega\cap B_{r/2}.
\end{equation}
Throughout this section, we denote $D_r=B_{\varrho r}(x_r)$.

We will divide the proof of Theorem \ref{thm-Lip} in several steps.
First, we have the following boundary Harnack type estimate, which is an immediate consequence of Theorem \ref{thm-main}.

\begin{lem}\label{lem1}
Let $s\in(0,1)$ and $\Omega\subset\R^n$ be any open set.
Assume that there is $x_0\in B_{1/2}$ and $\varrho>0$ such that $B_{2\varrho}(x_0)\subset \Omega\cap B_{1/2}$.
Denote $D=B_\varrho(x_0)$.

Then, there exists is $\delta>0$, depending only on $n$, $s$, $\varrho$, and ellipticity constants, such that the following statement holds.

Let $u_1$ and $u_2$ be two functions satisfying, for all $a,b\in\R$,
\begin{equation}\label{pb3}
\left\{ \begin{array}{rcll}
M^+(au_1+bu_2) &\geq&-|a|C_0-|b|\delta&\textrm{in }B_1\cap \Omega\\
u_1=u_2&=&0&\textrm{in }B_1\setminus\Omega,
\end{array}\right.\end{equation}
with $u_1,u_2\geq0$ in $\R^n$ and $\inf_D u_2=1$.
Then,
\begin{equation}\label{mec00}
\inf_D \frac{u_1}{u_2}\leq C\left(\inf_{B_{1/2}}\frac{u_1}{u_2}+C_0\right).
\end{equation}
The constant $C$ depends only on $n$, $s$, $\varrho$, and ellipticity constants.
\end{lem}

\begin{proof}
Dividing by $\inf_D u_1$ if necessary, we may assume $\inf_D u_1=1$.

By the interior Harnack inequality, $1=\inf_D u_2\leq \sup_D u_2\leq C$ (provided that $\delta$ is small enough).
Thus,
\[\inf_D \frac{u_1}{u_2}\leq C_1,\]
with $C_1$ independent of $C_0$.

Now, if $C_0\leq \delta$, then by Theorem \ref{thm-main} we have $u_2\leq C_2u_1$ in $B_{1/2}$, and therefore
\[\inf_D \frac{u_1}{u_2}\leq C_1\leq C_1C_2\left(\inf_{B_{1/2}}\frac{u_1}{u_2}\right).\]
If $C_0\geq\delta$, then we simply have
\[\inf_D \frac{u_1}{u_2}\leq C_1\leq \frac{C_1}{\delta}C_0=CC_0.\]
In any case, \eqref{mec00} is proved.
\end{proof}

Second, we need the following consequence of the interior Harnack.

\begin{lem}\label{lem2}
Let $s\in(0,1)$ and $\Omega\subset\R^n$ be any open set.
Assume that there is $x_0\in B_{1/2}$ and $\varrho>0$ such that $B_{2\varrho}(x_0)\subset \Omega\cap B_{1/2}$.
Denote $D=B_\varrho(x_0)$.

Then, there exists is $\delta>0$, depending only on $n$, $s$, $\varrho$, and ellipticity constants, such that the following statement holds.

Let $u_1$ and $u_2$ be two functions satisfying $u_1,u_2\geq0$ in $\R^n$, \eqref{pb3}, and $\inf_D u_2=1$.
Then,
\begin{equation}\label{mec2}
\sup_D \frac{u_1}{u_2}\leq C\left(\inf_D\frac{u_1}{u_2}+C_0\right).
\end{equation}
The constant $C$ depends only on $n$, $s$, $\varrho$, and ellipticity constants.
\end{lem}

\begin{proof}
Notice that $M^+u_1\geq-C_0$ and $M^-u_1\leq C_0$ in $\Omega\cap B_1$, while $M^+u_2\geq-\delta$ and $M^-u_2\leq \delta$ in $\Omega\cap B_1$.

By interior Harnack inequality, we have $1=\inf_D u_2\leq \sup_D u_2\leq C$ (provided that $\delta$ is small enough).
Moreover, for $u_1$ we have $\sup_D u_1\leq C(\inf_D u_1+C_0)$, and thus
\[ \sup_D \frac{u_1}{u_2}\leq C\sup_D u_1\leq C\left(\inf_D u_1+C_0\right)\leq C\left(\inf_D\frac{u_1}{u_2}+C_0\right),\]
as desired.
\end{proof}

We will also need the following rescaled versions of the previous Lemmas.

\begin{cor}\label{cor1}
Let $s\in(0,1)$, $r\in(0,1)$, and $\Omega\subset\R^n$ be any Lipschitz domain, with $0\in\partial\Omega$.
Then, there exists is $\delta>0$, depending only on $n$, $s$, $\varrho$ in \eqref{P}, ellipticity constants, such that the following statement holds.

Let $u_1$ and $u_2$ be two functions satisfying, for all $a,b\in\R$,
\begin{equation}\label{pb7}
\left\{ \begin{array}{rcll}
M^+(au_1+bu_2) &\geq&-|a|K-|b|\delta/C_1&\textrm{in }B_r\cap \Omega\\
u_1=u_2&=&0&\textrm{in }B_r\setminus\Omega,
\end{array}\right.\end{equation}
with $C_1>0$ and $u_1,u_2\geq0$ in $\R^n$.
Assume in addition that
\begin{equation}\label{asst}
\frac{r^{2s}}{\inf_{D_r}u_2}\leq C_1.
\end{equation}
Then,
\begin{equation}\label{mec}
\inf_{D_r} \frac{u_1}{u_2}\leq C\left(\inf_{B_{r/2}}\frac{u_1}{u_2}+K\,\frac{r^{2s}}{\inf_{D_r}u_2}\right).
\end{equation}
The constant $C$ depends only on $n$, $s$, $\varrho$, and ellipticity constants.
\end{cor}

\begin{proof}
The functions $v_1(x):=u_1(rx)/\inf_{D_r}u_2$ and $v_2(x):=C_1 u_2(rx)/\inf_{D_r}u_2$ satisfy
\[\left\{ \begin{array}{rcll}
M^+(av_1+bv_2) &\geq&-|a|K\frac{r^{2s}}{\inf_{D_r}u_2}-|b|\delta &\textrm{in }B_1\cap \Omega\\
v_1=v_2&=&0&\textrm{in }B_1\setminus\Omega.
\end{array}\right.\]
Thus, the result follows from Lemma \ref{lem1}.
\end{proof}

\begin{cor}\label{cor2}
Let $s\in(0,1)$, $r\in(0,1)$, and $\Omega\subset\R^n$ be any Lipschitz domain, with $0\in\partial\Omega$.
Then, there exists is $\delta>0$, depending only on $n$, $s$, $\varrho$ in \eqref{P}, and ellipticity constants, such that the following statement holds.

Let $u_1$ and $u_2$ be two functions satisfying $u_1,u_2\geq0$ in $\R^n$, and \eqref{pb7}.
Assume in addition \eqref{asst}.
Then,
\begin{equation}\label{mec2}
\sup_{D_r} \frac{u_1}{u_2}\leq C\left(\inf_{D_r}\frac{u_1}{u_2}+K\,\frac{r^{2s}}{\inf_{D_r}u_2}\right).
\end{equation}
The constant $C$ depends only on $n$, $s$, $\varrho$, and ellipticity constants.
\end{cor}

\begin{proof}
Setting $v_1(x):=u_1(rx)/\inf_{D_r}u_2$ and $v_2(x):=C_1u_2(rx)/\inf_{D_r}u_2$, the result follows from Lemma \ref{lem2}.
\end{proof}

We will also need the following.

\begin{lem}\label{lem-Lip-dom}
Let $s\in(0,1)$ and $\Omega\subset\R^n$ be any Lipschitz domain, with $0\in\partial\Omega$.
There exists is $\delta>0$, $\gamma\in(0,1)$, and $c_0>0$ depending only on $n$, $s$, $\Omega$, and ellipticity constants, such that the following statement holds.

Let $u$ be a viscosity solution of $M^+u\geq -\delta$ and $M^-u\leq \delta$ in $B_1\cap \Omega$, with $u=0$ in $B_1\setminus\Omega$.
Assume in addition that $u\geq0$ in $\R^n$ and $\inf_{D_1}u=1$.

Then, $u\geq c_0d^{2s-\gamma}$ in $B_{1/2}$,
where $d(x)={\rm dist}(x,B_1\setminus\Omega)$.
In particular,
\[\inf_{D_r}u\geq c_0r^{2s-\gamma}\qquad \textrm{for all}\quad r\in(0,1).\]
The constants $\gamma$ and $c_0$ depend only on $n$, $s$, $\Omega$, and ellipticity constants.
\end{lem}

\begin{proof}
We differ the proof to the Appendix.
\end{proof}

As a consequence, we find the following.

\begin{cor}\label{cor7}
Let $s\in(0,1)$ and $\Omega\subset\R^n$ be any Lipschitz domain, with $0\in\partial\Omega$.
There exists is $\delta>0$, depending only on $n$, $s$, $\Omega$, and ellipticity constants, such that the following statement holds.

Let $u_2$ be a viscosity solution of $M^+u_2\geq -\delta$ and $M^-u_2\leq \delta$ in $B_1\cap \Omega$, with $u_2=0$ in $B_1\setminus\Omega$.
Assume in addition that $u_2\geq0$ in $\R^n$.

Then, there is $\gamma\in(0,1)$ such that
\[\sup_{B_{2r|z|}}u_2\leq C|z|^{2s-\gamma}\inf_{D_r}u_2\qquad \textrm{whenever}\quad |z|\geq\frac12\quad \textrm{and}\quad r|z|\leq \frac14.\]
The constants $\gamma$ and $C$ depend only on $n$, $s$, $\Omega$, and ellipticity constants.
\end{cor}

\begin{proof}
We use the previous Lemma with
\[v(x):=\frac{u_2(4r|z|x)}{\inf_{D_{4r|z|}}u_2},\]
to find
\[c|z|^{\gamma-2s}=t^{2s-\gamma} \leq C\inf_{D_t}v = C\frac{\inf_{D_r} u_2}{\inf_{D_{4r|z|}} u_2},\]
where $t=\frac14|z|^{-1}$.
Thus,
\[\inf_{D_{4r|z|}} u_2 \leq  C|z|^{2s-\gamma}\inf_{D_r} u_2.\]
Moreover, by Lemma \ref{lem-use} we have
\[\sup_{B_{2r|z|}}u_2\leq C\inf_{D_{4r|z|}} u_2,\]
then
\[\sup_{B_{2r|z|}}u_2\leq C|z|^{2s-\gamma}\inf_{D_r} u_2,\]
and we are done.
\end{proof}

Using the previous results, we now prove the following.

\begin{lem}\label{lem-main}
Let $s\in(0,1)$ and $\Omega\subset\R^n$ be any Lipschitz domain, with $0\in \partial\Omega$.
Then, there exists $\delta>0$, depending only on $n$, $s$, $\varrho$ in \eqref{P}, and ellipticity constants, such that the following statement holds.

Let $u_1,u_2\in C(B_1)$ be viscosity solutions \eqref{pb} satisfying \eqref{u-is-nonneg}.
Then,
\begin{equation}\label{mec0}
\sup_{\Omega\cap B_r} \frac{u_1}{u_2}-\inf_{\Omega\cap B_r}\frac{u_1}{u_2}\leq Cr^\alpha
\end{equation}
for all $r\leq 3/4$.
The constants $C$ and $\alpha\in(0,1)$ depend only on $n$, $s$, $\varrho$, and ellipticity constants.
\end{lem}

\begin{proof}
We will prove that there exist constants $C_1>0$ and $\alpha>0$, and monotone sequences $\{m_k\}_{k\geq1}$ and $\{\bar m_k\}_{k\geq1}$, such that
\[\bar m_k-m_k=4^{-\alpha k},\quad 0\leq m_k\leq m_{k+1}<\bar m_{k+1}\leq \bar m_k\leq 1,\]
and
\begin{equation}\label{C1}
m_ku_2\leq C_1^{-1}u_1\leq \bar m_k u_2\quad \textrm{in}\ B_{r_k}, \qquad r_k=4^{-k}.
\end{equation}
Clearly, if such sequences exist, then \eqref{mec0} holds for all $r\leq \frac14$.
We will construct such sequences inductively.

First notice that, by Theorem \ref{thm-main} (and a covering argument), we have
\begin{equation}\label{bound}
0\leq u_1\leq \tilde C_1u_2\quad \textrm{in}\ B_{3/4},
\end{equation}
for some constant $\tilde C_1$.
Thus, it follows that \eqref{mec0} holds for $\frac14\leq r\leq \frac34$, and that we may take $m_1=0$, $\bar m_1=1$.
Furthermore, by taking $C_1\geq \tilde C_1 4^{\alpha k_0}$ we see that \eqref{C1} holds with for all $k\leq k_0$, with $m_k=0$ and $\bar m_k=4^{-\alpha k}$ for $1\leq k\leq k_0$, and $k_0$ is to be chosen later.

Assume now that we have sequences up to $m_k$ and $\bar m_k$ (with $k\geq k_0$), and let
\[v_k:=C_1^{-1}u_1-m_ku_2.\]
Notice that by induction hypothesis we have $v_k\geq0$ in $B_{r_k}$ (but not in all of $\R^n$).
Moreover, since $C_1^{-1}u_1\geq m_j u_2$ in $B_{r_j}$ for $j\leq k$, then
\[v_k\geq (m_j-m_k)u_2\geq (m_j-\bar m_j+\bar m_k-m_k)u_2=-(4^{-\alpha j}-4^{-\alpha k})u_2\quad \textrm{in}\ B_{r_j},\]
for every $j\leq k$.
Using now that for every $x\in B_1\setminus B_{r_k}$ there is $j<k$ such that $|x|<r_j=4^{-j}\leq 4|x|$, we find
\[v_k(x)\geq -u_2(x)\bigl(|4x|^\alpha-r_k^\alpha\bigr)\quad \textrm{in}\ B_{1/4}\setminus B_{r_k}.\]
Thanks to this, and since $v_k\geq0$ in $B_{r_k}$, for every $x\in B_{r_k/2}$ we have that the negative part of $v_k$ satisfies
\[\begin{split}
0&\leq M^- v_k^-(x)\leq M^+ v_k^-(x)=\Lambda\int_{x+y\notin B_{r_k}} v_k^-(x+y)\frac{dy}{|y|^{n+2s}}\\
& \leq C\int_{\frac{r_k}{2}\leq |y|\leq \frac14} u_2(x+y)\bigl(|4y|^\alpha-r_k^\alpha\bigr)\frac{dy}{|y|^{n+2s}}
        +\int_{\R^n\setminus B_{1/4}}C_1^{-1}u_1(x+y)\frac{dy}{|y|^{n+2s}}\\
&= Cr_k^{\alpha-2s} \int_{\frac12\leq |z|\leq \frac{1}{4r_k}}
    \frac{\bigl(|4z|^\alpha-1\bigr)u_2(x+r_k z)}{|z|^{n+2s}}\,dz +
        CC_1^{-1}\int_{\R^n}u_1(y)\frac{dy}{1+|y|^{n+2s}}\\
&\le Cr_k^{\alpha-2s} \int_{\frac12\leq |z|\leq \frac{1}{4r_k}}
    \frac{\bigl(|4z|^\alpha-1\bigr)\sup_{B_{2r_k|z|}}u_2}{|z|^{n+2s}}\,dz + CC_1^{-1}.
\end{split}\]

Now, by Corollary \ref{cor7} there is $\gamma>0$ such that
\[\sup_{B_{2r_k|z|}}u_2\leq C|z|^{2s-\gamma}\bigl(\inf_{D_{r_k}}u_2\bigr)\]
for every $|z|\geq\frac12$ and $r|z|\leq \frac14$, and thus
\[\begin{split}
Cr_k^{\alpha-2s}\int_{\frac12\leq |z|\leq \frac{1}{4r_k}}\frac{\bigl(|4z|^\alpha-1\bigr)\sup_{B_{2r_k|z|}}u_2}{|z|^{n+2s}}\,dz&\leq
 Cr_k^{\alpha-2s}\bigl(\inf_{D_{r_k}}u_2\bigr) \int_{\frac12\leq |z|\leq \frac{1}{4r_k}} \frac{\bigl(|4z|^\alpha-1\bigr)|z|^{2s-\gamma}}{|z|^{n+2s}}\,dz \\
&\leq \varepsilon_0 r_k^{\alpha-2s}\bigl(\inf_{D_{r_k}}u_2\bigr),
\end{split}\]
with
\[\varepsilon_0:=C\int_{|z|\geq \frac12}\frac{\bigl(|4z|^\alpha-1\bigr)|z|^{2s-\gamma}}{|z|^{n+2s}}\,dz\longrightarrow0\qquad \textrm{as}\quad \alpha\rightarrow0.\]
This means that
\[0\leq M^- v_k^-\leq M^+ v_k^-\leq \varepsilon_0 r_k^{\alpha-2s}\bigl(\inf_{D_{r_k}}u_2\bigr)+CC_1^{-1}\quad\textrm{in}\ B_{r_k/2}.\]
Therefore, since $v_k^+=C_1^{-1}u_1-m_ku_2+v_k^-$, we have
\[\begin{split} M^-v_k^+&\leq C_1^{-1}M^-(u_1-m_ku_2)+M^+v_k^-\leq C_1^{-1}(1+m_k)\delta+\varepsilon_0r_k^{\alpha-2s}\bigl(\inf_{D_{r_k}}u_2\bigr)+CC_1^{-1}\\
&\leq \delta+\varepsilon_0r_k^{\alpha-2s}\bigl(\inf_{D_{r_k}}u_2\bigr)+CC_1^{-1}\end{split}\]
in $\Omega\cap B_{r_k/2}$.
Also,
\[M^+v_k^+\geq M^+v_k\geq -(C_1^{-1}+m_k)\delta\geq-\delta\quad \textrm{in}\ \Omega\cap B_{r_k/2}.\]
Similarly, we have
\[M^+(a v_k^+ + bu_2)\geq -|a|\left(\delta+\varepsilon_0r_k^{\alpha-2s}\bigl(\inf_{D_{r_k}}u_2\bigr)+CC_1^{-1}\right)-|b|\delta \qquad\textrm{in}\quad \Omega\cap B_{r_k/2}.\]

Now, recall that by Corollary \ref{cor7} we have
\[\frac{r_k^{2s}}{\inf_{D_{r_k}}u_2}\leq Cr_k^\gamma\leq C_1.\]
Thus, we can apply Corollaries \ref{cor1} and \ref{cor2} to the functions $v_k^+$ and $u_2$, to obtain
\[\begin{split}
\inf_{D_{r_k}}\frac{v_k^+}{u_2}& \leq
   C\inf_{B_{r_k/2}}\frac{v_k^+}{u_2}+ C\left(\delta+\varepsilon_0r_k^{\alpha-2s}\bigl(\inf_{D_{r_k}}u_2\bigr)+CC_1^{-1}\right) \frac{r_k^{2s}}{\inf_{D_{r_k}}u_2} \\
&\leq C\inf_{B_{r_k/2}}\frac{v_k^+}{u_2}+ C(\delta+C_1^{-1})r_k^\gamma+C\varepsilon_0 r_k^\alpha,\end{split}\]
and
\[\sup_{D_{r_k/2}}\frac{v_k^+}{u_2}\leq C\inf_{D_{r_k/2}} \frac{v_k^+}{u_2}+C(\delta+C_1^{-1})r_k^\gamma+C\varepsilon_0r_k^\alpha.\]
Recalling that $v_k^+=v_k=C_1^{-1}u_1-m_ku_2$ in $B_{r_k/2}$, we find
\[\inf_{D_{r_k/2}}(C_1^{-1}u_1/u_2-m_k)\leq C\inf_{B_{r_k/4}}(C_1^{-1}u_1/u_2-m_k)+C(\delta+C_1^{-1})r_k^\gamma +C\varepsilon_0r_k^\alpha,\]
and
\[\sup_{D_{r_k/2}}(C_1^{-1}u_1/u_2-m_k)\leq C\inf_{D_{r_k/2}} (C_1^{-1}u_1/u_2-m_k)+C(\delta+C_1^{-1}) r_k^\gamma +C\varepsilon_0r_k^\alpha.\]
Therefore,  we deduce
\[\sup_{D_{r_k/2}}(C_1^{-1}u_1/u_2-m_k)\leq C\inf_{B_{r_k/4}}(C_1^{-1}u_1/u_2-m_k)+C(\delta+C_1^{-1}) r_k^\gamma +C\varepsilon_0r_k^\alpha.\]

Repeating the same argument with $\bar v_k:=\bar m_k-C_1^{-1}u_1$ instead of $v_k$, we find
\[\sup_{D_{r_k/2}}(\bar m_k-C_1^{-1}u_1/u_2)\leq C\inf_{B_{r_k/4}}(\bar m_k-C_1^{-1}u_1/u_2)+C(\delta+C_1^{-1}) r_k^\gamma +C\varepsilon_0r_k^\alpha.\]
Thus, combining the previous estimates, we get
\[\begin{split}
\bar m_k-m_k&\leq C\inf_{B_{r_k/4}}(C_1^{-1}u_1/u_2-m_k)+C\inf_{B_{r_k/4}}(\bar m_k-C_1^{-1}u_1/u_2) + C(\delta+C_1^{-1}) r_k^\gamma +C\varepsilon_0r_k^\alpha \\
&=C\left(\inf_{B_{r_k/4}}(C_1^{-1}u_1/u_2)-\sup_{B_{r_k/4}}(C_1^{-1}u_1/u_2)+\bar m_k-m_k + (\delta+C_1^{-1}) r_k^\gamma +\varepsilon_0r_k^\alpha\right).
\end{split}\]
Using that $\bar m_k-m_k=4^{-\alpha k}$, $r_k=4^{-k}$, and $k\geq k_0$, we obtain
\[\sup_{B_{r_{k+1}}}(C_1^{-1}u_1/u_2)-\inf_{B_{r_{k+1}}}(C_1^{-1}u_1/u_2)\leq \left(\frac{C-1}{C}+(\delta+C_1^{-1}) 4^{-(\gamma-\alpha)k_0} +\varepsilon_0\right)4^{-\alpha k}.\]
Taking $\alpha$ small enough and $k_0$ large enough, we get
\[\sup_{B_{r_{k+1}}}(C_1^{-1}u_1/u_2)-\inf_{B_{r_{k+1}}}(C_1^{-1}u_1/u_2)\leq 4^{-\alpha (k+1)}.\]
This means that we can choose $m_{k+1}$ and $\bar m_{k+1}$, and thus we are done.
\end{proof}

We finally give the:

\begin{proof}[Proof of Theorem \ref{thm-Lip}]
We will combine Lemma \ref{lem-main} with interior estimates in order to get the desired result.

Let $x,y\in \Omega\cap B_{1/2}$, let
\[r=|x-y|\qquad \textrm{and}\qquad d=\min\{d(x),d(y)\},\]
where $d(x)=\textrm{dist}(x,\partial\Omega)$.
Let $x_*\in\partial\Omega$ be such that $d(x)=|x-x_*|$.
We need to show that $\bigl|(u_1/u_2)(x)-(u_1/u_2)(y)\bigr|\leq Cr^{\alpha'}$, with $\alpha'>0$.
Since $u_1/u_2$ is bounded in $B_{3/4}$, we may assume that $0<r\leq r_0$, with $r_0$ small enough.

If $r\leq d/2$, then by interior estimates \cite{CS} we have
\[\|u_i\|_{C^\alpha(B_{d/2}(x))}\leq Cd^{-\alpha}.\]
Since $\inf_{B_{d/2}(x)} u_2\geq c_0d^{2s-\gamma}$, then
\[\|u_2^{-1}\|_{C^\alpha(B_{d/2}(x))}\leq Cd^{\gamma-\alpha-2s}.\]
Therefore, for $r\leq d/2$ we have
\[\bigl|(u_1/u_2)(x)-(u_1/u_2)(y)\bigr|\leq Cr^\alpha d^{\gamma-2\alpha-2s}\leq Cr^\alpha d^{-2s}.\]
provided that $\alpha\leq \gamma/2$.
In particular, if $r\leq d^\theta/2$, with $\theta>2s/\alpha>1$, then
\begin{equation}\label{ineq1}
\bigl|(u_1/u_2)(x)-(u_1/u_2)(y)\bigr|\leq Cr^{\alpha-2s/\theta}.
\end{equation}

On the other hand, for all $r\in(0,r_0)$ we have $x,y\in B_{d+r}(x_*)$, and thus by Lemma~\ref{lem-main} we have
\[\bigl|(u_1/u_2)(x)-(u_1/u_2)(y)\bigr|\leq \sup_{B_{d+r}(x_*)\cap\Omega}\frac{u_1}{u_2}-\inf_{B_{d+r}(x_*)\cap\Omega}\frac{u_1}{u_2}\leq C(d+r)^\alpha.\]
In particular, if $r\geq d^\theta/2$ then
\begin{equation}\label{ineq2}
\bigl|(u_1/u_2)(x)-(u_1/u_2)(y)\bigr|\leq Cr^{\theta\alpha}.
\end{equation}

Combining \eqref{ineq1} and \eqref{ineq2}, we find
\[\bigl|(u_1/u_2)(x)-(u_1/u_2)(y)\bigr|\leq Cr^{\alpha'}\qquad \textrm{for all}\quad r\in(0,1),\]
with $\alpha'=\min\{\alpha-2s/\theta,\,\theta\alpha\}>0$.
Thus, the Theorem is proved.
\end{proof}

\section{Non-symmetric operators with drift}
\label{sec5}

The above proofs of Theorems \ref{thm-main} and Theorem \ref{thm-Lip} work as well for operators of the form
\[\tilde Lu(x)=\int_{\R^n}\bigl(u(x+y)-u(x)-\nabla u(x)\cdot y\chi_{B_1}(y)\bigr)K(x,y)dy+b(x)\cdot\nabla u,\]
provided that $s\geq\frac12$.
Namely, consider the class of nonlocal and non-symmetric operators
\begin{equation}\label{drift1}
\tilde Lu(x)=\int_{\R^n}\bigl(u(x+y)-u(x)-\nabla u(x)\cdot y\chi_{B_1}(y)\bigr)K(y)dy+b\cdot\nabla u,
\end{equation}
with $K$ satisfying \eqref{ellipt-const-linear} and
\begin{equation}\label{drift2}
|b|+\left|r^{2s-1}\int_{B_1\setminus B_r} y\,K(y)dy\right|\leq \beta.
\end{equation}
Given $\lambda$, $\Lambda$, and $\beta$, we define the class $\mathcal L(\lambda,\Lambda,\beta)$ as the set of all linear operators \eqref{drift1} satisfying \eqref{ellipt-const-linear} and \eqref{drift2}.
Then, we may define $\widetilde M^\pm$ as
\[\widetilde M^+ u=\widetilde M^+_{\mathcal L(\lambda,\Lambda,\beta)}u=\sup_{\tilde L\in\mathcal L(\lambda,\Lambda,\beta)} \tilde Lu,\qquad
 \widetilde M^- u=\widetilde M^-_{\mathcal L(\lambda,\Lambda,\beta)}u=\inf_{\tilde L\in\mathcal L(\lambda,\Lambda,\beta)} \tilde Lu.\]

For such operators, Theorems \ref{half-Harnack-sub} and \ref{half-Harnack-sup} were established in \cite{CD}; see Corollaries 4.3 and 6.2 therein.
Using such results, and with the exact same proofs given in the previous Sections, we find the following.

\begin{thm}\label{thm-main-drift}
Let $s\in[\frac12,1)$ and $\Omega\subset\R^n$ be any open set.
Assume that there is $x_0\in B_{1/2}$ and $\varrho>0$ such that $B_{2\varrho}(x_0)\subset \Omega\cap B_{1/2}$.

Then, there exists $\delta>0$, depending only on $n$, $s$, $\varrho$, $\lambda$, $\Lambda$, and $\beta$, such that the following statement holds.

Let $u_1,u_2\in C(B_1)$ be viscosity solutions of
\begin{equation}\label{pb-drift}
\left\{ \begin{array}{rcll}
\widetilde M^+(au_1+bu_2) &\geq&-\delta(|a|+|b|)&\textrm{in }B_1\cap \Omega\\
u_1=u_2&=&0&\textrm{in }B_1\setminus\Omega
\end{array}\right.\end{equation}
for all $a,b\in\R$, and such that
\begin{equation}\label{u-is-nonneg-drift}
u_i\geq0\quad\mbox{in}\quad \R^n, \qquad \int_{\R^n}\frac{u_i(x)}{1+|x|^{n+2s}}\,dx=1.
\end{equation}
Then,
\[ C^{-1}u_2\leq u_1\leq C\,u_2\qquad\textrm{in}\ B_{1/2}.\]
The constant $C$ depends only on $n$, $s$, $\varrho$, $\lambda$, $\Lambda$, and $\beta$.
\end{thm}

Moreover, we also have the following.

\begin{thm}\label{thm-Lip-drift}
Let $s\in[\frac12,1)$ and $\Omega\subset\R^n$ be any Lipschitz domain, with $0\in\partial\Omega$.
Then, there is $\delta>0$, depending only on $n$, $s$, $\Omega$, $\lambda$, $\Lambda$, and $\beta$, such that the following statement holds.

Let $u_1,u_2\in C(B_1)$ be viscosity solutions of \eqref{pb-drift} satisfying \eqref{u-is-nonneg-drift}.
Then, there is $\alpha\in(0,1)$ such that
\[ \left\|\frac{u_1}{u_2}\right\|_{C^{0,\alpha}(\overline\Omega\cap B_{1/2})}\leq C.\]
The constants $\alpha$ and $C$ depend only on $n$, $s$, $\Omega$, $\lambda$, $\Lambda$, and $\beta$.
\end{thm}

To our best knowledge, Theorems \ref{thm-main-drift} and \ref{thm-Lip-drift} are new even for the linear operator $(-\Delta)^{1/2}+b\cdot\nabla$.
Those results will be used in the forthcoming paper \cite{FR-drift}.

%
%

\section{Appendix: Subsolution in Lipschitz domains}

We prove here a lower bound for positive solutions $u$ in Lipschitz domains, namely $u\geq cd^{2s-\gamma}$ in $\Omega$ for some small $\gamma>0$.
This is stated in Lemma \ref{lem-Lip-dom}, which we prove below.

For this, we need to construct the following subsolution.

\begin{lem}\label{homog-subsol}
Let $s\in(0,1)$, and $e\in S^{n-1}$.
Given $\eta>0$, there is $\epsilon>0$ depending only on $n$, $s$, $\eta$ and ellipticity constants such that the following holds.

Define
\[\Phi(x) := \left( e\cdot x- \eta |x| \left(1- \frac{(e\cdot x)^2}{|x|^2} \right)\right)_+^{2s-\epsilon}\]
Then,
\[
\begin{cases}
M^-  \Phi  \ge  0 \quad & \mbox{in }\mathcal C_\eta  \\
\Phi = 0  \quad & \mbox{in }\R^n \setminus \mathcal C_\eta
\end{cases}
\]
where $\mathcal C_{\eta}$ is the cone defined by
\[\mathcal C_{\eta}: = \left\{ x \in \R^n\ :  e\cdot \frac{x}{|x|}    >  \eta \left( 1 -  \left( e\cdot \frac{ x }{|x|} \right)^2\right) \right\}.\]
The constant $\epsilon$ depends only on $\eta$, $s$, and ellipticity constants.

In particular $\Phi$ satisfies $M^-\Phi\ge 0$ in all of $\R^n$.
\end{lem}

\begin{proof}
By homogeneity it is enough to prove that, for $\epsilon$ small enough, we have $M^-\Phi \ge 1$ on points belonging to $e + \partial \mathcal{  C}_\eta$, since all the positive dilations of this set with respect to the origin cover the interior of $\mathcal{\tilde  C}_\eta$.

Let thus $P\in  \partial \mathcal{ C}_\eta$, that is,
\[ e\cdot P- \eta \left( |P| -  \frac{(e\cdot P)^2}{|P|} \right) =0.\]

Consider
\[\begin{split}
\Phi_{P}(x) &:= \Phi(P+e+x)
 \\
&= \left( e\cdot (P+e+x)- \eta \left( |P+e+x| -  \frac{(e\cdot (P+e+x))^2}{|P+e+x|} \right)\right)_+^{2s-\epsilon}
\\
&= \left( 1 +e\cdot x- \eta \left( |P+e+x| -|P|-  \frac{(e\cdot (P+e+x))^2}{|P+e+x|} +\frac{(e\cdot P)^2}{|P|} \right)\right)_+^{2s-\epsilon}
\\
&=\bigl( 1 +e\cdot x- \eta \psi_P(x)  \bigr)_+^{s+\epsilon},
\end{split}\]
where we define
\[\psi_P(x)  :=  |P+e+x| -|P|-  \frac{(e\cdot (P+e+x))^2}{|P+e+x|} +\frac{(e\cdot P)^2}{|P|} .\]

Note that the functions $\psi_P$  satisfy
\[
 |\nabla \psi_P(x)|  \le  C  \quad \mbox{in } \R^n \setminus \{ -P-e\},
\]
and
\begin{equation} \label{esthess}
|D^2 \psi_P(x)|  \le  C  \quad \mbox{for  }x \in B_{1/2},
\end{equation}
where $C$ does not depend on $P$ (recall that $|e|=1$).

Now for fixed $\tilde e \in  \partial \mathcal{ C}_\eta\cap \partial B_1$ let us compute
\[
\lim_{t\uparrow +\infty} \psi_{t\tilde e}(x)  =  \lim_{t\uparrow +\infty} (|t\tilde e +e+x|-|t\tilde e|) -  \lim_{t\uparrow +\infty}  \left(\frac{(e\cdot (t\tilde e +e+x))^2}{|t\tilde e +e+x|} -\frac{(e\cdot t \tilde e)^2}{|t\tilde e |} \right).
\]

On the one hand, we have
\[
\lim_{t\uparrow +\infty} (|\tilde e t+e+x|-|\tilde e t|)  =  \tilde e\cdot(e+x).
 \]

 On the other hand to compute for $f_t(y) := \frac{(e\cdot (t\tilde e + y))^2}{|t\tilde e +y|} $ we have
\[
\partial_{y_i} f_t(y) = \frac{2(e\cdot (t\tilde e+ y)) e_i}{|t\tilde e +y|} -  \frac{(e\cdot (t\tilde e + y))^2}{|t\tilde e +y|^3} (t\tilde e +y )_i
\]
and hence
\[
 \lim_{t\uparrow +\infty} \partial_{y_i} f_t(y)    = \big(2(e\cdot \tilde e) e_i- (e\cdot\tilde e)^2 \tilde e _i\big).
\]
Therefore,
\[
 \lim_{t\uparrow +\infty} \left(\frac{(e\cdot (t\tilde e +e+x))^2}{|t\tilde e +e+x|}-\frac{(e\cdot t \tilde e)^2}{|t\tilde e |}\right) = \big(2(e\cdot \tilde e)e - (e\cdot\tilde e)^2\tilde e\big) \cdot(e+x) .
\]

We have thus found
\[
\lim_{t\uparrow +\infty} \psi_{P}(x)  =   \big(\tilde e-2(e\cdot \tilde e)e + (e\cdot\tilde e)^2\tilde e\big)\cdot(e+x)
\]
and
\[
\lim_{t\uparrow +\infty} \big(1+e\cdot x-\eta \psi_{P}(x)\big)  =   \big(e -\eta \tilde e+ 2\eta(e\cdot \tilde e)e-\eta (e\cdot\tilde e)^2\tilde e\big)\cdot(e+x)
\]

Note that for $\delta$ small enough (depending only on $\eta$), if we define
\[ \mathcal C_{\tilde e}:= \left\{x\in \R^n \ :\  \frac{x+e}{|x+e|}\cdot \frac{e-(e\cdot\tilde e)\tilde e}{|e-(e\cdot\tilde e)\tilde e|} \ge (1-\delta)\right\} \]
satisfies
\begin{equation}\label{Plarge}
\lim_{t\uparrow +\infty} \big(1+e\cdot x-\eta \psi_{P}(x)\big) \ge c|x|\quad \mbox{for all }x\in \mathcal C_{\tilde e}
\end{equation}
where $c>0$. Indeed, the vector $e' :=e-(e\cdot\tilde e)\tilde e$ is perpendicular to $\tilde e$ and has positive scalar product with $e$. Thus, we have
\[
\big(e -\eta \tilde e+ 2\eta(e\cdot \tilde e)e-\eta (e\cdot\tilde e)^2\tilde e\big)\cdot e'>0
\]

Let us show now that  for $\varepsilon>0$ small enough the function $\Phi_P$ satisfies
\begin{equation} \label{goal}
M^-\Phi_P(0) \ge 1.
\end{equation}

We first prove \eqref{goal} in the case $|P|\ge R$ with $R$ large enough.
Indeed let $P= t\tilde e$ for $t\uparrow +\infty$ and $\tilde e \in  \partial \mathcal{ C}_\eta\cap \partial B_1$.
Let us denote
\[
\delta^2 u(x,y) = \frac{u(x+y)+u(x-y)}{2}-u(x).
\]
Using \eqref{esthess}, and \eqref{Plarge}, and  $\Phi_{P} \ge 0$ we obtain
\[
\begin{split}
\lim_{t\to \infty} M^-\Phi_{P} (0)
&\ge
\int_{\R^n} \left( (\delta^2 u)_+ \frac{\lambda}{|y|^{n+2s}} -  (\delta^2 u)_- \frac{\Lambda}{|y|^{n+2s}}\right) \,dy
\\
&\ge
 \int_{ \mathcal C_{\tilde e}} (c|y| -C)_+^{2s-\epsilon}\,\frac{dy}{|y|^{n+2s}}  - C  \int_{\R^n} \min\{1,|y|^2\} \frac{dy}{|y|^{n+2s}}
\\
&\ge \frac{c}{\epsilon}-C.
\end{split}
\]

Thus \eqref{goal} follows for $|P|\ge R$ with $R$ large, provided that $\epsilon$ is taken small enough.

We now concentrate in the case $|P|<R$. In this case we use that, taking $\delta>0$ small enough (depending on $\eta$) and defining the cone
\[ \mathcal C_{e}:= \left\{x\in \R^n \ :\  \frac{x}{|x|}\cdot e \ge (1-\delta)\right\} \]
we have
\[
e\cdot (P+e+x)- \eta \left( |P+e+x| -  \frac{(e\cdot (P+e+x))^2}{|P+e+x|} \ge c |x|\right)
\]
for $x\in \mathcal C_e$ with $|x|\ge L$ with $L$ large enough (depending on $R$).

Thus, reasoning similarly as above but now integrating in $\mathcal C_e \cap \{ |x|>L\}$ instead of on  $\mathcal C_{\tilde e}$ we prove
\eqref{goal} also in the case $P\ge R$, provided that $\epsilon$ is small enough. Therefore the lemma is proved.
\end{proof}

Finally, we give the:

\begin{proof}[Proof of Lemma \ref{lem-Lip-dom}]
Note that we only need to prove the conclusion of the Lemma for $r>0$ small enough, since the conclusion for non-small $r$ follows from the interior Harnack inequality.

Recall that $\Omega\subset\R^n$ is assumed to Lipschitz domain, with $0\in\partial\Omega$.
Then, for some $e\in S^{n-1}$, $\eta>0$ (typically large), and ${r_0}>0$ depending on (the Lipschitz regularity of) $\Omega$ we have
\[
\tilde{\mathcal C}_\eta \cap B_{2r_0}\subset\Omega
\]
where $\tilde {\mathcal C}_\eta$ is the cone of Lemma \ref{homog-subsol}, which is very sharp for $\eta$ large.

Let $\Phi$ and $\epsilon>0$ be the subsolution and the constant in Lemma \ref{homog-subsol}. We now take
\[
\tilde \Phi = \big( \Phi -(|x|/r_0)^2 \big) \chi_{2r_0}.
\]
By Lemma \eqref{homog-subsol} we have
\[
M^-\tilde \Phi \ge - C \quad \mbox{in }B_{r_0}
\]
while clearly $\tilde \Phi\le 0$ outside $B_{r_0}$.

Now we take observe that, for $c_1>0$ small enough we have
\[
M^-(c_1\tilde\Phi + \chi_{D_1}) \ge -c_1 C   + c \ge c/2>0
\]
in $B_{r_0}$ --- not that $B_{r_0}\cap D_1 = \varnothing$ since $r_0$ is small.

Then, taking $\delta\in(0,c/2)$
we have
\[
M^-(u-c_1\tilde\Phi + \chi_{D_1}) \le 0 \quad \mbox{ in }B_{r_0}
\]
while
\[
u-c_1\tilde\Phi + \chi_{D_1} \ge 0-c_1\tilde\Phi +0 \ge 0\quad\mbox{in}(\R^n\setminus B_{r_0})\setminus {D_1}
\]
and
\[
u-c_1\tilde\Phi + \chi_{D_1}  = (u-1)-c_1\tilde\Phi \ge 0-c_1\tilde\Phi \ge 0\quad \mbox{in } (\R^n\setminus B_{r_0})\cap {D_1}.
\]

Then, by the maximum principle we obtain
\[
u-c_1\tilde\Phi = u-c_1\tilde\Phi + \chi_{D_1} \ge 0 \quad \mbox{in } B_{r_0}
\]
and hence
\[
u(x)\ge c_1\Phi(x) - C |x|^2 \quad\mbox{for }x\in B_{r_0}
\]
which clearly implies the Lemma (taking $\gamma=\epsilon$).
\end{proof}



\begin{thebibliography}{00}



\bibitem[BB94]{BHP3} R. Bass, K. Burdzy, \emph{The boundary Harnack principle for non-divergence form elliptic operators}, J. Lond. Math. Soc. 50 (1994), 157-169.

\bibitem[BL02]{BL} R. Bass, D. Levin, \emph{Harnack inequalities for jump processes}, Potential Anal. 17 (2002), 375-382.

\bibitem[Bog97]{Bogdan1} K. Bogdan, \emph{The boundary Harnack principle for the fractional Laplacian}, Studia Math., 123 (1997), 43-80.

\bibitem[BKK08]{Bogdan2} K. Bogdan, T. Kulczycki, and M. Kwasnicki, \emph{Estimates and structure of $\alpha$-harmonic functions}, Probab. Theory Related Fields 140 (2008), 345-381.

\bibitem[BKK15]{Bogdan3} K. Bogdan, T. Kumagai, M. Kwasnicki, \emph{Boundary Harnack inequality for Markov processes with jumps}, Trans. Amer. Math. Soc. 367 (2015), 477-517.


\bibitem[CS09]{CS} L. Caffarelli, L. Silvestre, \emph{Regularity theory for fully nonlinear integro-differential equations}, Comm. Pure Appl. Math. 62 (2009), 597-638.

\bibitem[CS11]{CS3} L. Caffarelli, L. Silvestre, \emph{The Evans-Krylov theorem for nonlocal fully nonlinear equations}, Ann. of Math. 174 (2011), 1163-1187.


\bibitem[CD16]{CD} H. Chang-Lara, G. Davila, \emph{H\"older estimates for nonlocal parabolic equations with critical drift}, J. Differential Equations 260 (2016), 4237-4284.



\bibitem[FR16]{FR-drift} X. Fernandez-Real, X. Ros-Oton, \emph{The obstacle problem for the fractional Laplacian with critical drift}, preprint arXiv (Oct. 2016).

\bibitem[RS15]{RS-C1}  X. Ros-Oton, J. Serra, \emph{Boundary regularity estimates for nonlocal elliptic equations in $C^1$ and $C^{1,\alpha}$ domains}, preprint arXiv (Dec. 2015).


\bibitem[Sil06]{S} L. Silvestre, \emph{H\"older estimates for solutions of integro differential equations like the fractional Laplacian}, Indiana Univ. Math. J. 55 (2006), 1155-1174.

\bibitem[SW99]{SW} R. Song, J.-M. Wu, \emph{Boundary Harnack principle for symmetric stable processes}, J. Funct. Anal. 168 (1999), 403-427.

\end{thebibliography}
\end{document}